\documentclass[12pt]{article}
\usepackage{amsmath,amssymb,amsthm, amsfonts}
\usepackage{etaremune, enumerate, float, verbatim}
\usepackage{algorithm}
\usepackage{algorithmic}
\usepackage{hyperref}
\usepackage{graphicx}
\usepackage{url}
\usepackage[mathlines]{lineno}
\usepackage{dsfont} 
\usepackage{tikz, graphicx, subcaption, caption}
\usepackage[algo2e,ruled,vlined]{algorithm2e}
\usepackage{mathrsfs,amsmath}
\usepackage{color}
\definecolor{red}{rgb}{1,0,0}

\definecolor{blue}{rgb}{0,0,.7}

\definecolor{green}{rgb}{0,.6,0}

\definecolor{purp}{rgb}{.5,0,.5}

\numberwithin{figure}{section}   

\newtheorem{thm}{Theorem}
\newtheorem{cor}[thm]{Corollary}

\newtheorem{prop}[thm]{Proposition}

\theoremstyle{definition}

\theoremstyle{definition}

\theoremstyle{definition}

\newcommand{\Z}{\operatorname{Z}}

\newcommand{\zbar}{\overline{\Z}}

\newcommand{\bit}{\begin{itemize}}
\newcommand{\eit}{\end{itemize}}
\newcommand{\ben}{\begin{enumerate}}
\newcommand{\een}{\end{enumerate}}
\newcommand{\beq}{\begin{equation}}
\newcommand{\eeq}{\end{equation}}
\newcommand{\bea}{\begin{eqnarray*}} 
\newcommand{\eea}{\end{eqnarray*}}
\newcommand{\bpf}{\begin{proof}}
\newcommand{\epf}{\end{proof}\ms}
\newcommand{\bmt}{\begin{bmatrix}}
\newcommand{\emt}{\end{bmatrix}}
\newcommand{\ms}{\medskip}

\newcommand{\noi}{\noindent}

\title{Minimal Zero Forcing Sets}
\author{Boris Brimkov
\thanks{Department of Mathematics and Statistics, Slippery Rock University, Slippery Rock, PA, USA (boris.brimkov@sru.edu)}\and Joshua Carlson
\thanks{Department of Mathematics and Computer Science, Drake University, Des Moines, IA, USA (joshua.carlson@drake.edu)}
}
\date{}

\begin{document}
\maketitle
\begin{abstract}
In this paper, we study minimal (with respect to inclusion) zero forcing sets. We first investigate when a graph can have polynomially or exponentially many distinct minimal zero forcing sets. We also study the maximum size of a minimal zero forcing set $\zbar(G)$, and relate it to the zero forcing number $\Z(G)$. Surprisingly, we show that the equality $\zbar(G)=\Z(G)$ is preserved by deleting a universal vertex, but not by adding a universal vertex. We also characterize graphs with extreme values of $\zbar(G)$ and explore the gap between $\zbar(G)$ and $\Z(G)$.
\end{abstract}

\noi {\bf Keywords} Zero forcing; minimal zero forcing set; universal vertex

\noi{\bf AMS subject classification} 05C15, 05C57, 05C76

\section{Introduction}


Given a simple undirected graph $G=(V,E)$, each of whose vertices is colored blue or white, the \emph{zero forcing color change rule} says that at each timestep, a blue vertex with exactly one white neighbor causes  that neighbor to become blue. If $S\subset V$ is a set of blue vertices in $G$, the \emph{closure} of $S$, denoted $\emph{cl}(S)$, is the set of blue vertices obtained after the color change rule is applied until no more white vertices can be turned blue. A set $S$ is a \emph{zero forcing set} if $cl(S)=V$; the \emph{zero forcing number} of $G$, denoted $\Z(G)$, is the minimum cardinality of a zero forcing set.  

Zero forcing was introduced in \cite{AIM-Workshop} as a bound on the minimum rank over all symmetric matrices whose entries have the same off-diagonal zero-nonzero pattern as the adjacency matrix of a graph $G$. This minimum rank problem is a special case of the matrix completion problem which has numerous theoretical and practical applications (such as the million-dollar Netflix challenge \cite{netflix}). Zero forcing is also related to other processes that arise from the observation that knowing the values of all-but-one variables in a linear equation causes the last remaining variable to be known. In particular, processes that are equivalent or very similar to zero forcing were independently introduced in quantum physics (quantum control theory \cite{quantum1}), theoretical computer science (fast-mixed searching \cite{fast_mixed_search}), electrical engineering (PMU placement \cite{BruneiHeath,powerdom3}), and combinatorial optimization (target set selection problem \cite{target1,target3,target2}). Zero forcing has also found a variety of uses in physics, logic circuits, coding theory, and in modeling the spread of diseases and information in social networks; see \cite{zf_tw,quantum1,logic1,zf_np} and the bibliographies therein.

In this paper, we study minimal (rather than minimum) zero forcing sets. Specifically, we study when a graph can have polynomially or exponentially many distinct minimal zero forcing sets. We also investigate the maximum size of a minimal zero forcing set, and relate it to the zero forcing number. Maximum minimal sets and minimum maximal sets have been studied in the context of many other graph parameters, including independent sets (see \cite{indep2, indep1}), dominating sets (see \cite{minmaxdom2,minmaxdom1}), matchings (see \cite{minmaxmatching1,minmaxmatching2}), and vertex covers (see \cite{vertCover1,vertCover2}). Studying minimal zero forcing sets can lead to a better understanding of the zero forcing process, e.g., in the context of zero forcing polynomials \cite{zfp} and zero forcing reconfiguration graphs \cite{geneson}.

This paper is organized as follows. In the remainder of this section, we recall some graph theoretic notions, specifically those related to zero forcing. In Section \ref{sec:number}, we study the effect of various graph properties on the number of minimal zero forcing sets. In Section \ref{sec:max_min}, we investigate the maximum size of a minimal zero forcing set and its relation to the zero forcing number. We conclude with some final remarks and open questions in Section \ref{conclusion}.

\subsection{Preliminaries}
\label{sec:Preliminaries}

A simple graph $G=(V,E)$ consists of a vertex set $V$ and an edge set $E$ of two-element subsets of $V$.  The \emph{order} of $G$ is denoted by $n=|V|$. Two vertices $v,w\in V$ are \emph{adjacent}, or \emph{neighbors}, if $\{v,w\}\in E$; this is denoted $v\sim w$. The \emph{neighborhood} of $v\in V$ is the set of all vertices which are adjacent to $v$, denoted $N(v)$; the \emph{closed neighborhood} of $v$, denoted $N[v]$, is the set $N(v)\cup\{v\}$. The \emph{degree} of $v\in V$ is defined as $\deg(v)=|N(v)|$. The minimum degree of $G$ is denoted $\delta(G)$ and the maximum degree is denoted $\Delta(G)$. A \emph{leaf} is a vertex of degree 1 and a \emph{universal vertex} is a vertex of degree $|V|-1$, i.e., a vertex that is adjacent to all other vertices. Given $S \subset V$, the \emph{induced subgraph} $G[S]$ is the subgraph of $G$ whose vertex set is $S$ and whose edge set consists of all edges of $G$ which have both endpoints in $S$. 

The complete graph on $n$ vertices is denoted by $K_n$, the cycle on $n$ vertices is denoted by $C_n$, and the empty graph on $n$ vertices is denoted by $\overline{K}_n$. A \emph{tree} is a connected acyclic graph. A \emph{branchpoint} of a tree is a vertex of degree at least 3. A tree with a single branchpoint is called a \emph{spider}. The branchpoint $v$ of a spider is also called the \emph{center} vertex, and the \emph{legs} of a spider $G$ with branchpoint $v$ are the connected components of $G-v$. The \emph{length} of a leg of a spider is the number of vertices in the leg. The graph $S_{a_1,\ldots,a_k}$ is a spider whose legs have lengths $a_1,\ldots,a_k$. The \emph{corona} of graphs $G$ and $H$, denoted $G\circ H$, is the graph obtained by joining all vertices of a copy of $H$ with each vertex of $G$. The \emph{join} of disjoint graphs $G$ and $H$, denoted $G\vee H$, is the graph obtained by adding an edge between every vertex of $G$ and every vertex of $H$. The \emph{wheel} on $n$ vertices is defined as $W_n=C_{n-1}\vee K_1$ and the \emph{star} on $n$ vertices is defined as $S_n=\overline{K}_{n-1}\vee K_1$. 

A connected component of $G$ is called \emph{trivial} if it consists of a single vertex; otherwise it is called  \emph{nontrivial}. The disjoint union of graphs $G_1$ and $G_2$ is denoted $G_1 \dot\cup G_2$, and $kG = \dot\bigcup_{i=1}^k G$. An \emph{isomorphism} between graphs $G$ and $H$ is a bijection $f\colon V(G)\to V(H)$ such that vertices $u$ and $v$ are adjacent in $G$ if and only if $f(u)$ and $f(v)$ are adjacent in $H$. If graphs $G$ and $H$ are isomorphic, we will write $G\cong H$. An \emph{automorphism} is an isomorphism from $G$ to itself. A graph $G$ is \emph{vertex transitive} if for any two vertices $v_1$ and $v_2$ of $G$, there is some automorphism $f\colon G\to G$
such that $f(v_{1})=v_{2}$. For other graph theoretic terminology and definitions, we refer the reader to~\cite{west}.  

A \emph{fort} of a graph $G=(V,E)$ is a non-empty set $F\subset V$ such that no vertex outside $F$ is adjacent to exactly one vertex in $F$. It was shown in \cite{brimkov_zf1} that every zero forcing set of a graph intersects every fort of the graph.
The set of all forts of $G$ is denoted as $\mathcal{B}(G)$; when there is no scope for confusion, dependence on $G$ will be omitted. For a zero forcing set $S \subset V(G)$, an ordered list of forces that can be performed in sequence to color $V(G)$ blue is called a \emph{chronological list of forces} of $S$. Given an arbitrary subset $S \subset V(G)$, a set of forces that can be performed (in some order) to color $cl(S)$ blue is called a \emph{set of forces of $S$}. For a set of forces, $\mathcal{F}$, of $S \subset V(G)$, the \emph{terminus of $\mathcal{F}$} is the set of vertices in $V(G)$ that do not perform a force in $\mathcal{F}$. The terminus of an arbitrary set of forces of a subset $S \subset V(G)$ is called a \emph{reversal of $S$}. It is easy to see that every reversal of a zero forcing set $S$ is also a zero forcing set of the same size as $S$. The maximum size of a minimal zero forcing set of $G$ is denoted $\zbar(G)$. \\

\section{Number of minimal zero forcing sets}\label{sec:number}
In this section, we investigate the effect (or lack thereof) of several graph properties on the number of minimal zero forcing sets. We begin with a general characterization using $\zbar(G)$.

\begin{prop}
If $\zbar(G)=O(1)$, then $G$ has a polynomial number of minimal zero forcing sets. If $\zbar(G)=\Omega(n)$, then $G$ could have either a polynomial or an exponential number of minimal zero forcing sets.
\end{prop}
\proof
For every minimal zero forcing set $S$ of $G$, $|S|\leq \zbar(G)$. There are ${n \choose 1}+{n \choose 2}+\ldots+{n \choose \zbar(G)}$ subsets of $V(G)$ of size at most $\zbar(G)$. If $\zbar(G)=k$ for some constant $k$, then ${n \choose 1}+{n \choose 2}+\ldots+{n \choose \zbar(G)}=O(n^k)$ and therefore, the number of minimal zero forcing sets of $G$ is polynomial. Next, $\zbar(K_n)=\Omega(n)$ and $K_n$ has polynomially many minimal zero forcing sets, since each minimal zero forcing set consists of $n-1$ vertices of $K_n$. Finally, $\zbar(K_{n/2}\circ K_1)=\Omega(n)$ and $K_{n/2}\circ K_1$ has exponentially many minimal zero forcing sets, since any set $S$ consisting of $n/4$ leaves and $n/4$ non-leaves that are not adjacent to any of the $n/4$ leaves is zero forcing (as each leaf in $S$ can force its white neighbor, and then each non-leaf in $S$ can force its white neighbor) and minimal (as deleting any element of $S$ will cause it to not be a zero forcing set).
\qed
\vspace{9pt}

\noindent Next we show that the number of minimal zero forcing sets is not determined by whether the graph is a tree. 

\begin{prop}
\label{prop_trees}
Some trees have polynomially many minimal zero forcing sets; some trees have exponentially many minimal zero forcing sets. 
\end{prop}
\proof
Let $S_{5,5,\ldots,5}$ be a spider with center $v$ and $\frac{n-1}{5}$ legs of length 5. For $1\leq i\leq\frac{n-1}{5}$, let leg $i$ consist of vertices $a_i$, $b_i$, $c_i$, $d_i$, and $e_i$, where $a_i$ is adjacent to $v$, $b_i$ to $a_i$, $c_i$ to $b_i$, $d_i$ to $c_i$, and $e_i$ to $d_i$. For each $I\subset \{1,\ldots,\frac{n-1}{5}\}$ and $j\in \{1,\ldots,\frac{n-1}{5}\}$, the set  $S(I,j)=\bigcup_{i\in I, i\neq j}\{b_i,c_i\}\cup\bigcup_{i\notin I,i\neq j}\{c_i,d_i\}$ is zero forcing set, since in each leg $i$ different from $j$, either the vertices $\{b_i,c_i\}$ or $\{c_i,d_i\}$ are contained in the set, and those vertices will force the entire leg $i$; then, after all legs $i\neq j$ are colored blue, leg $j$ will be forced by the center $v$. Moreover, the set $S(I,j)$ is minimal, since if any vertex $u$ in leg $i\neq j$ is omitted, the leg $i$ cannot be colored. Thus, there are $\Omega(2^{n/5})$ minimal zero forcing sets in $S_{5,5,\ldots,5}$. On the other hand, a star $S_n$ has a polynomial number of minimal zero forcing sets, since any minimal zero forcing set of $S_n$ consists of all-but-one leaves. 
\qed
\vspace{9pt}

\noindent Next we show that the number of minimal zero forcing sets in a tree is not determined by the number of leaves or branchpoints. 

\begin{prop}
A tree with an exponential number of minimal zero forcing sets can have the same number of leaves, branchpoints, and vertices as a tree with a polynomial number of minimal zero forcing sets. 
\end{prop}
\proof
Let $S_{1,1,\ldots,1,(4n+1)/5}$ be a spider with $\frac{n-6}{5}$ legs of length 1 and one leg of length $\frac{4n+1}{5}$. This spider has a total of $\frac{n-1}{5}$ legs, and therefore $\frac{n-1}{5}$ leaves. The minimal zero forcing sets of this graph consist of either all-but-one leaves, or of all-but-one of the leaves in the legs of length 1 plus two adjacent non-leaf, non-center vertices in the leg of length $\frac{4n+1}{5}$. Thus, $S_{1,1,\ldots,1,(4n+1)/5}$ has polynomially many minimal zero forcing sets. On the other hand, in Proposition \ref{prop_trees}, it was shown that $S_{5,5,\ldots,5}$ has exponentially many minimal zero forcing sets, yet it has the same number of vertices, leaves, and branchpoints as 
$S_{1,1,\ldots,1,(4n+1)/5}$. 
\qed

\vspace{9pt}
\noindent We next show a direct relation between the number of connected components and the number of minimal zero forcing sets.
\begin{prop}
Let $G_1,\ldots,G_k$ be the connected components of a graph $G$. For $1\leq i\leq k$, let $n_i$ be the number of minimal zero forcing sets of $G_i$. Then, the number of minimal zero forcing sets of $G$ is $\prod_{i=1}^kn_i$.
\end{prop}
\proof
A set $S$ is a zero forcing set of $G$ if and only if $S\cap V(G_i)$ is a zero forcing set of $G_i$ for each $1\leq i\leq k$. Moreover, $S$ is minimal if and only if $S\cap V(G_i)$ is minimal for each $1\leq i\leq k$. Thus, each minimal zero forcing set $S$ of $G$ corresponds to a collection of minimal zero forcing sets of $G_1,\ldots,G_k$. Furthermore, since there are $\prod_{i=1}^kn_i$ distinct ways to select minimal zero forcing sets of $G_1,\ldots,G_k$, it follows there are $\prod_{i=1}^kn_i$ distinct minimal zero forcing sets in $G$.
\qed
\begin{cor}
\label{corcomp}
If a graph has $k$ nontrivial components, then it has $\Omega(2^k)$ minimal (and minimum) zero forcing sets. 
\end{cor}
\proof
Let $G_1,\ldots,G_k$ be the nontrivial connected components of $G$. For $1\leq i\leq k$, let $S_i^1$ be a minimum zero forcing set of $G_i$ and $S_i^2$ be a reversal of $S_i^1$. Then, for each $I\subset \{1,\ldots,k\}$, the set $S(I)=\bigcup_{i\in I} S_i^1 \cup \bigcup_{i\notin I} S_i^2$ is a minimum (and hence minimal) zero forcing set of $G$. Thus, $G$ has $\Omega(2^k)$ minimal zero forcing sets.  
\qed

\vspace{9pt}

\noindent We conclude this section by investigating whether vertex transitivity can affect the number of minimal zero forcing sets. By Proposition \ref{corcomp}, it is easy to see that a disconnected vertex transitive graph can have both a polynomial and exponential number of minimal zero forcing sets. For example, any disjoint union of vertex transitive graphs of fixed size, like $\dot\bigcup_{i=1}^{n/3}C_3$, has exponentially many zero forcing sets. However, $C_{n/2}\dot\cup C_{n/2}$ has polynomially many minimal zero forcing sets. Next, we show that this also holds for connected vertex transitive graphs.

\begin{prop}
\label{propproduct}
Some connected vertex transitive graphs have exponentially many minimal zero forcing sets; some have polynomially many minimal zero forcing sets.
\end{prop}
\proof
Let $G=C_5 \square K_{n/5}$. Note that $G$ is a connected vertex transitive graph with $\Z(G)=2n/5$. Let $K^1$, $K^2$, and $K^3$ be three distinct maximal cliques of $G$ with $V(K^1)=\{u_1,\ldots,u_{n/5}\}$, $V(K^2)=\{v_1,\ldots,v_{n/5}\}$, and $V(K^3)=\{w_1,\ldots,w_{n/5}\}$ such that $v_i\sim u_i$ and $v_i \sim w_i$ for all $i\in \{1,\ldots,n/5\}$. 

For each $I\subset \{1,\ldots,n/5\}$, let $K^1(I)=\{u_i:i\in I\}$ and $K^3(I)=\{w_i:i\in \{1,\ldots,n/5\}\backslash I\}$. Then, $S(I):=V(K^2)\cup K^1(I)\cup K^3(I)$ is a zero forcing set of $G$, since for each $i\in I$, $v_i$ can force $w_i$ and for each $i\in \{1,\ldots,n/5\}\backslash I$, $v_i$ can force $u_i$. After every vertex in $K^1$, $K^2$, and $K^3$ is colored blue, the rest of the graph can also be forced. See Figure \ref{Fig_exp_zf_set} for an illustration. Since $\Z(G)=2n/5$ and $|S(I)|=2n/5$ for each $I\subset \{1,\ldots,n/5\}$, $S(I)$ is a minimum (and hence also minimal) zero forcing set. There are $2^{n/5}$ subsets $I$ of $\{1,\ldots,n/5\}$, and each of them creates a distinct minimum zero forcing set $S(I)$; thus, there are $\Omega(2^{n/5})$ distinct minimum zero forcing sets of $G$.

\begin{figure}[h]
\centering
\scalebox{1}{\includegraphics{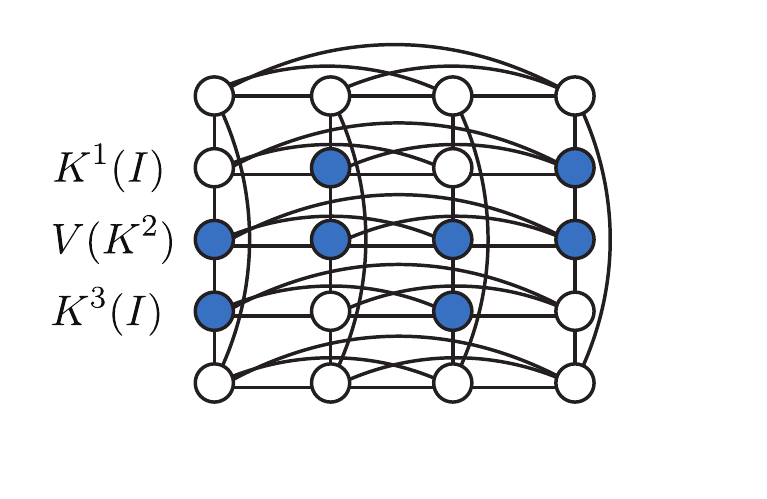}}
\caption{A connected vertex transitive graph with an exponential number of minimal zero forcing sets.}
\label{Fig_exp_zf_set}
\end{figure}

\noindent Conversely, the cycle $C_n$ is a connected vertex transitive graph with $n$ minimal zero forcing sets, since any set containing two adjacent vertices is a zero forcing set, and any set consisting of more than two adjacent vertices is not minimal.
\qed

\vspace{9pt}

\noindent While vertex transitive graphs with a polynomial number of minimal zero forcing sets can be both sparse and dense (e.g., cycles and complete graphs), and vertex transitive graphs with an exponential number of minimal zero forcing sets can be dense (e.g., the family shown in Proposition \ref{propproduct}), we have not found a sparse family of vertex transitive graphs with an exponential number of minimal zero forcing sets. We leave this as an open question.

%
%
%
%


\section{Maximum minimal zero forcing sets}
\label{sec:max_min}
In addition to exploring the number of minimal zero forcing sets in a graph, it is also interesting to consider the possible sizes of minimal zero forcing sets in the graph. In this section, we explore $\zbar(G)$, the maximum size of a minimal zero forcing set of $G$. We begin by characterizing the extremal values of $\zbar(G)$. It is easy to see that for a graph $G$ of order $n$, $\zbar(G) \leq n$, with equality holding if and only if $G$ is the empty graph $\overline{K_n}$. Thus, the first nontrivial extremal value to consider is $\zbar(G) = n - 1$.

\begin{prop}
Let $G$ be a graph on $n$ vertices. Then $\zbar(G)=n-1$ if and only if $G\cong K_m \cup kK_1$ where $k\geq 2$ is an integer and $m = n - k$.
\label{zbarnminus1}
\end{prop}
\begin{proof}

Let $G$ be a graph with $\zbar(G)=n-1$ and suppose $G\not\cong K_m \cup kK_1$ where $k\geq 2$ is an integer and $m = n - k$.
Suppose first that $G$ has multiple nontrivial components. In this case, a minimal zero forcing set of $G$ cannot contain all vertices from some nontrivial component of $G$. Thus, each minimal zero forcing set contains at most $n-2$ vertices, so $\zbar(G)\leq n-2$, a contradiction.
Now, suppose $G$ has a single nontrivial component $C$. By the assumption that $G\not\cong K_m \cup kK_1$, it follows that $C$ is not a clique. This means that there must be two non-adjacent vertices $u$ and $v$ in $C$. 
Let $S$ be a minimal zero forcing set of $G$ of size $n-1$. Then, $S$ must contain all isolated vertices of $G$, and therefore it contains all-but-one vertices of $C$. 

Suppose one of $u$ and $v$, say $u$, is not in $S$. Let $w$ be a neighbor of $v$. Then, $S\setminus \{w\}$ is also a zero forcing set; this contradicts the minimality of $S$. Thus, both $u$ and $v$ have to be in $S$ which means there is some other vertex $w \notin \{u,v\}$ that is not in $S$. If $w$ is not a dominating vertex of $C$, then there is a vertex $q$ not adjacent to $w$. Let $p$ be a neighbor of $q$. Then, $S\setminus \{p\}$ is also a zero forcing set of $G$. If $w$ is a dominating vertex, then $S\setminus \{v\}$ is also a zero forcing set of $G$, since $u$ can force $w$, and then $w$ can force $v$. Thus, if $\zbar(G)=n-1$, $G$ must be isomorphic to $K_m \cup kK_1$. Conversely, if $G\cong K_m \cup kK_1$, it is easy to see that $\zbar(G)=n-1$.
\end{proof}




\noindent Next, we consider low values of $\zbar(G)$. Since for any graph $G$, $\zbar(G) \geq \Z(G)$, we turn our attention to characterizing $\zbar(G) = \Z(G)$.
\begin{prop} \label{prop:zbarEqualsZ}
Let $G$ be a graph. Then, $\zbar(G)=Z(G)$ if and only if every zero forcing set of $G$ contains a minimum zero forcing set. 
\end{prop}
\begin{proof}
Suppose $\zbar(G)=\Z(G)$ and suppose there exists a zero forcing set $S$ of $G$ that does not contain a minimum zero forcing set. Let $S'$ be a minimal zero forcing set contained in $S$. Then, $\zbar(G)\geq |S'|>\Z(G)$, a contradiction. Conversely, if every zero forcing set contains a minimum zero forcing set, then every minimal zero forcing set must also contain a minimum zero forcing set and must therefore be a minimum zero forcing set.  
\end{proof}
\noindent The condition of Proposition \ref{prop:zbarEqualsZ} is easy to verify for some families of graphs, especially those with high symmetry. A few such families are given in the following corollary.
\begin{cor}
\label{cor:ZequalsZbar}
If $G$ is a cycle, empty graph, star, wheel, or complete graph, then $\zbar(G)=\Z(G)$.
\end{cor}
\noindent Below are two more families of graphs that satisfy $\zbar(G)=\Z(G)$.
\begin{prop}
\label{prop:ZequalsZbar}
For any integers $a,b\geq 3$, $\Z(K_a\vee \overline{K}_b)=\zbar(K_a\vee \overline{K}_b)$ and $\Z(K_a\vee C_b)=\zbar(K_a\vee C_b)$.
\end{prop}
\begin{proof}
In $K_a \vee \overline{K}_b$, each pair of vertices in $K_a$ and each pair of vertices in $\overline{K}_b$ form a fort. Thus, any zero forcing set of $K_a \vee \overline{K}_b$ must contain at least $a-1$ vertices of $K_a$ and at least $b-1$ vertices of $\overline{K}_b$. Moreover, each set consisting of exactly $a-1$ vertices of $K_a$ and exactly $b-1$  vertices of $\overline{K}_b$ is a minimum zero forcing set. Thus, by Proposition \ref{prop:zbarEqualsZ}, $\Z(K_a\vee \overline{K}_b)=\zbar(K_a\vee \overline{K}_b)$. 

By \cite[Lemma 4.3]{cospectral}, $\Z(K_a\vee C_b)=\min\{a+\Z(C_b),b+\Z(K_a)\}=\min\{a+2,b+a-1\}=a+2$. Let $S$ be an arbitrary zero forcing set of $K_a\vee C_b$. Note that $S$ must contain at least $a-1$ vertices of $K_a$, since each pair of vertices in $K_a$ is a fort. Suppose $S$ contains exactly $a-1$ vertices of $K_a$. Then, the first force cannot be performed by a vertex of $K_a$, since each vertex of $K_a$ will have at least two white neighbors. In order for a vertex of $C_b$ to perform the first force, it and its two neighbors in $C_b$ must be in $S$. However, a set consisting of $a-1$ vertices of $K_a$ and 3 consecutive vertices of $C_b$ is a minimum zero forcing set, so $S$ contains a minimum zero forcing set. 

Now, suppose $S$ contains $a$ vertices of $K_a$. Then, unless $b=3$, the first force still cannot be performed by a vertex of $K_a$, since each vertex of $K_a$ will have at least  two white neighbors. In order for a vertex of $C_b$ to perform the first force, it and one of its neighbors in $C_b$ must be in $S$. However, a set consisting of $a$ vertices of $K_a$ and 2 consecutive vertices of $C_b$ is a minimum zero forcing set, so again $S$ contains a minimum zero forcing set. Finally, if $b =3$, then $K_a \vee C_b$ is a complete graph. In all cases, by Proposition \ref{prop:zbarEqualsZ}, $\Z(K_a\vee C_b)=\zbar(K_a\vee C_b)$.

\end{proof}
\noindent In families of graphs without high symmetry, it can be difficult to determine whether every zero forcing set contains a minimum zero forcing set. Therefore, despite the complete characterization of $\zbar(G) = \Z(G)$ in Proposition \ref{prop:zbarEqualsZ}, the structure of these graphs is still unclear.

To obtain more insight about graphs with $\zbar(G) = \Z(G)$, we can look for graph operations that preserve this equality. From Corollary \ref{cor:ZequalsZbar} and Proposition \ref{prop:ZequalsZbar}, it seems like the operation of adding a universal vertex is a good candidate for preserving $\zbar(G) = \Z(G)$. In particular, adding any number of universal vertices to cycles, wheels, empty graphs, stars, complete graphs, and graphs of the form $K_a\vee \overline{K_b}$ and $K_a\vee C_b$ always produces another graph satisfying $\zbar(G) = \Z(G)$. However, the following result shows that there are graphs where adding a universal vertex does not preserve the property $\zbar(G) = \Z(G)$.

\begin{thm}\label{thrm:addingUniversals}
There are infinitely many graphs $G$ such that $\zbar(G) = \Z(G)$ and \\ $\zbar(G \vee K_1) > \Z(G \vee K_1)$.
\end{thm}

\begin{proof}
For each integer $n \geq 7$, let $G_n$ be the graph on $n$ vertices illustrated in Figure~\ref{fig:addUniversalCounterex}. Note that $\Z(G_n) \geq \delta(G_n) \geq 2$. Since $\{1,2,3\}$ is a zero forcing set of $G_n$ and no subset $S \subset V(G_n)$ with $|S| = 2$ is a zero forcing set, $\Z(G_n) = 3$. Let $H = G_n \vee K_1$. Since $G$ has no isolated vertices, $\Z(H) = \Z(G) + 1 = 4$ by \cite[Lemma 4.3]{cospectral}. Note that $\{1, 3, 4, 5, 6\}$ is a minimal zero forcing set of $H$ which implies that $\zbar(H) \geq 5 > 4 = \Z(H)$.

\begin{figure}[H]
\centering
\scalebox{.6}{\includegraphics{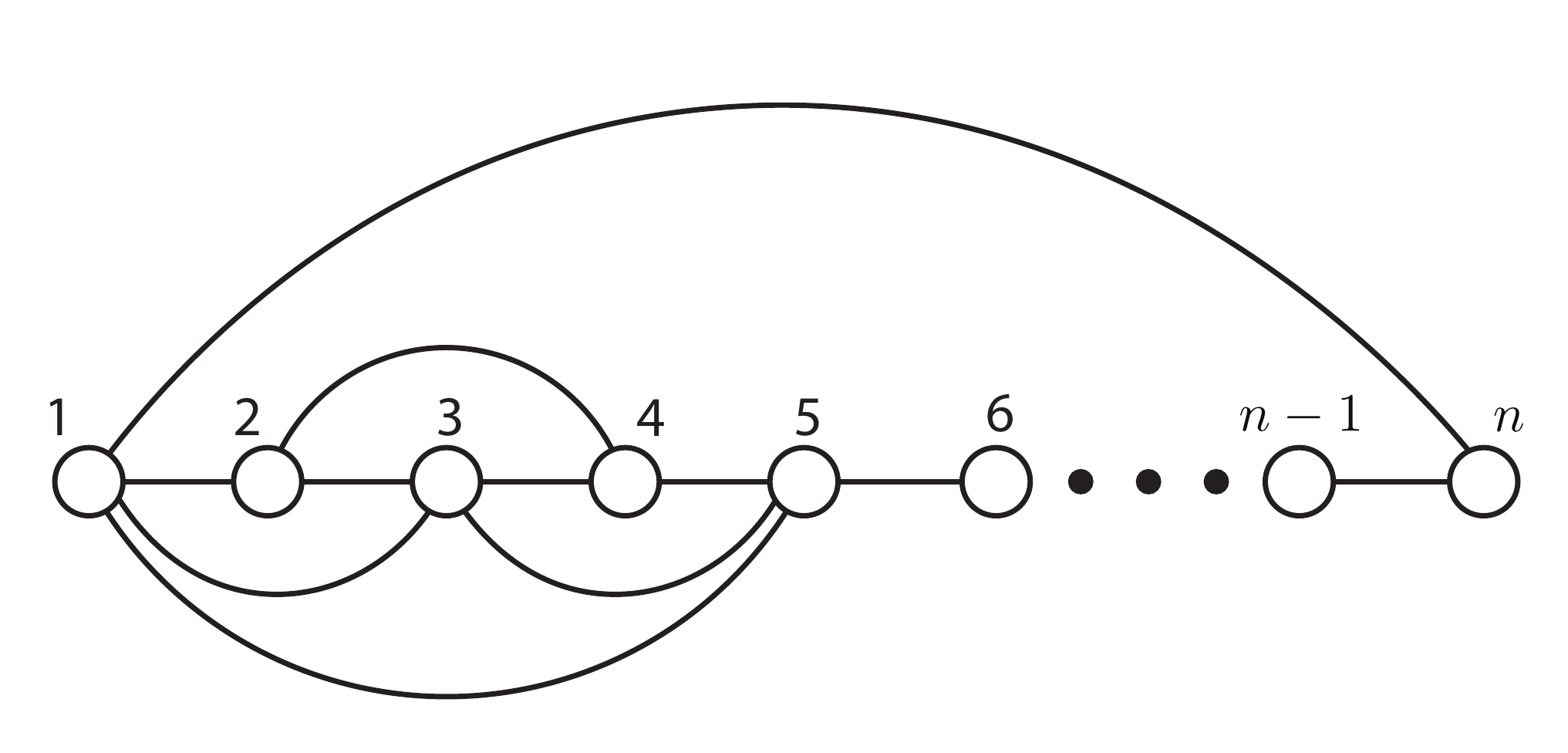}}
\caption{A graph $G_n$ on $n \geq 7$ vertices with $\zbar(G_n) = \Z(G_n)$ and $\zbar(G_n \vee K_1) > \Z(G_n \vee K_1)$.}\label{fig:addUniversalCounterex}
\end{figure}

It remains to show that $\zbar(G) =\Z(G) = 3$. Suppose that $B\subset V(G)$ is a minimal zero forcing set of G with $|B|\geq 4$. Let $v$ be the first vertex in $B$ to perform a force which means that $v$ and all-but-one of its neighbors are in $B$. If $v=1$, then $B$ contains one of $\{1,2,3,5\}$, $\{1,2,3,n\}$, $\{1,2,5,n\}$, and $\{1,3,5,n\}$ as a subset. These sets are zero forcing sets of $G$ that properly contain the following zero forcing sets respectively: $\{1,2,3\}$, $\{1,2,3\}$, $\{1,2,n\}$, $\{1,3,n\}$. 

Similarly, if $v=2$, then $B$ contains one of $\{1,2,3\}$, $\{1,2,4\}$, and $\{2,3,4\}$ as a subset (call this subset $X$). Since $|B| \geq 4$, $X$ is a proper subset of $B$; moreover, $X$ is a zero forcing set. If $v=3$, then $B$ contains one of $\{1,2,3,4\}$, $\{1,2,3,5\}$, $\{1,3,4,5\}$, and $\{2,3,4,5\}$ as a subset. Each of these sets are zero forcing sets of $G$ that properly contain another zero forcing set (namely, $\{1,2,3\}$ or $\{3,4,5\}$). Note that due to the symmetry of $G$, the cases where $v=4$ and $v=5$ are analogous to $v=2$ and $v=1$, respectively. 

Finally, if $v\in \{6,...,n\}$, then $B$ must contain a pair of vertices $\{x,y\}$ from the following list: $\{5,6\}, \{6,7\}, \dots, \{n-1,n\}, \{n,1\}$. Recall that $|B| \geq 4$. Since the vertices $1,5,6,7,\ldots,n$ are in the closure of $\{x,y\}$ and $B$ is a minimal zero forcing set, $B \setminus \{x,y\}$ contains at least two vertices from $\{2,3,4\}$. Thus, $B$ properly contains one of $\{x,y,2\}$, $\{x,y,3\}$, and $\{x,y,4\}$ which are each zero forcing sets of $G$. Therefore, in all cases, $B$ is not a minimal zero forcing set of $G$ which implies that $\zbar(G) \leq 3$. Since $\zbar(G) \geq \Z(G) = 3$, it follows that $\zbar(G) = 3$. \qedhere

\end{proof}

\noindent Although adding a universal vertex does not always preserve $\zbar(G) = \Z(G)$, the following result shows that deleting a universal vertex (if one exists) does in fact preserve $\zbar(G) = \Z(G)$.

\begin{thm}\label{thrm:deletingUniversals}
Let $G$ be a graph with a universal vertex $v$. If $\zbar(G) = \Z(G)$, then $\zbar(G-v) = \Z(G-v)$.
\end{thm}

\begin{proof}

We proceed by proving the contrapositive: if $\zbar(G)\neq \Z(G)$, then $\zbar(G\vee K_1)\neq \Z(G\vee K_1)$. Suppose $G$ is a graph with $\zbar(G) \neq \Z(G)$ and let $H = G \vee K_1$ where $V(K_1) = \{v\}$. Since $\zbar(G) \neq \Z(G)$, by Proposition \ref{prop:zbarEqualsZ}, $G$ has a zero forcing set $B \subset V(G)$ that does not contain a minimum zero forcing set of $G$. Let $B' \subset B$ be a minimal zero forcing set of $G$, where $|B'| > \Z(G)$. Note that $B' \cup \{v\}$ is a zero forcing set of $H$. 

Assume that $G$ has no isolated vertices. Then, $\Z(H) = \Z(G) + 1$, which means that $|B' \cup \{v\}| = |B'| + 1 > \Z(G) + 1 = \Z(H)$. Since $B'$ is a minimal zero forcing set of $G$, deleting vertices in $B'$ from $B' \cup \{v\}$ does not create a zero forcing set of $H$. Also, since $G$ has no isolated vertices and $B'$ is a minimal zero forcing set of $G$, $B'$ must be a proper subset of $V(G)$ and every vertex in $B'$ must have a neighbor in $G$ that is not in $B'$. Therefore, $B'$ is not a zero forcing set of $H$. Thus, $B' \cup \{v\}$ is a minimal zero forcing set of $H$ and $\zbar(H) >  \Z(H)$.

Next, assume that $G$ has exactly one isolated vertex $u$. Then, $\Z(H) = \Z(G)$ and $B'$ contains $u$. Since $u \in B'$, $B'$ is a zero forcing set of $H$ because $u$ can force $v$ which allows $B'$ to force the remaining vertices in $V(G)$. Since $B'$ is a minimal zero forcing set of $G$, every vertex in $B' \setminus \{u\}$ has a neighbor in $G$ that is not in $B'$. Thus, deleting $u$ from $B'$ does not create a zero forcing set of $H$. Now let $x \in B' \setminus \{u\}$. To see that $B' \setminus \{x\}$ is not a zero forcing set of $H$, note that the component $C$ of $G$ that contains $x$ has no isolated vertices. So by the previous case, $(B' \cap V(C)) \cup \{v\}$ is not a zero forcing set of $H[V(C) \cup \{v\}]$. This means that the vertices in $(V(C) \setminus B') \cup \{x\}$ contain a fort. Therefore, $B'$ is a minimal zero forcing set of $H$. Since, $|B'| > \Z(G) = \Z(H)$, $\zbar(H) > \Z(H)$.

Finally, assume that $G$ has at least two isolated vertices $u$ and $w$. In this case, $\Z(H) = \Z(G) - 1$. Note that $u,w \in B'$ and let $B'' = B' \setminus \{w\}$. To see that $B''$ is a zero forcing set of $H$, observe that $u$ can force $v$, $(B'\setminus \{w\})\cup \{v\}$ is a zero forcing set of $H - w$, and once $V(H) \setminus \{w\}$ is blue, $v$ can force $w$. Also, $|B''| = |B'| - 1 > \Z(G) - 1 = \Z(H)$. It remains to show that $B''$ is a minimal zero forcing set of $H$. First note that every pair $P = \{a,b\}$ of isolated vertices in $G$ is a fort of $H$ because no vertex in $H$ is adjacent to exactly one vertex in $P$. Thus, deleting an isolated vertex of $G$ from $B''$ does not create a zero forcing set of $H$. Similar to the previous cases, if $x \in B''$ is not an isolated vertex of $G$ and $C$ is the component of $G$ that contains $x$, then $(V(C) \setminus B'') \cup \{x\}$ contains a fort. Therefore, $B''$ is a minimal zero forcing set of $H$ which means $\zbar(H) > \Z(H)$.
\end{proof}

Theorems \ref{thrm:addingUniversals} and \ref{thrm:deletingUniversals} provide some interesting insight into the structure of the graphs that satisfy $\Z(G) = \zbar(G)$. For instance, we can define the poset $(\mathcal{G}, \preceq)$ where $\mathcal{G}$ is the set of graphs with $\Z(G) = \zbar(G)$ and for each $G,H \in \mathcal{G}$, $G \preceq H$ if and only if $H \cong G \vee K_1$. The poset $(\mathcal{G}, \preceq)$ could be a useful way to study the property $\Z = \zbar$. For example, consider the lengths of various chains in $(\mathcal{G}, \preceq)$. We have found many examples of infinitely long chains in this poset. The following graphs are minimal elements of such chains: $K_1$, $\overline{K}_n$, and $C_n$. On the other hand, since the graph $G_n$ in Figure \ref{fig:addUniversalCounterex} has no universal vertex, Theorem \ref{thrm:addingUniversals} also demonstrates that there are infinitely many chains in $(\mathcal{G}, \preceq)$ that only contain one graph each (nameley, $G_n$). Interestingly, we have not found a finite chain in $(\mathcal{G}, \preceq)$ with more than one graph and we leave this question open.

While it is difficult to give a full structural description of the graphs with $\zbar(G) = \Z(G)$, maximum minimal zero forcing sets can be described in terms of forts. Recall that a fort of a graph $G$ is a subset $S \subset V(G)$ such that no vertex in $V(G) \setminus S$ has exactly one neighbor in $S$. If $\mathcal{B}(G)$ is the collection of all forts of $G$, a \emph{cover} of $\mathcal{B}(G)$ is a set $S$ that intersects each fort in $\mathcal{B}(G)$. A \emph{minimal cover} of $\mathcal{B}(G)$ is a cover that does not contain another cover as a proper subset.

\begin{prop}
A set $S$ is a minimal zero forcing set of a graph $G=(V,E)$ if and only if $S$ is a minimal cover of $\mathcal{B}(G)$.
\end{prop}
\begin{proof}
Let $S$ be a minimal zero forcing set of $G$. It was shown in \cite{brimkov_zf1} that every zero forcing set intersects every fort, so $S$ is a cover of $\mathcal{B}(G)$. Suppose for contradiction that $S$ contains a smaller cover $S'$ as a proper subset. If $S'$ is not a zero forcing set, then $cl(S') \neq V$. If any vertex $u\in cl(S')$ is adjacent to exactly one vertex $v \in V\backslash cl(S')$, then $u$ could force $v$, contradicting the definition of $cl(S')$. Thus, $V \backslash cl(S')$ is a fort, and it does not contain any vertex of $S'$, which contradicts $S'$ being a cover. Therefore, $S'$ is a zero forcing set of $G$. However, this contradicts the assumption that $S$ is a minimal zero forcing set. 

Conversely, let $S$ be a minimal cover of $\mathcal{B}(G)$. It was shown in \cite{brimkov_zf1} that $S$ is a zero forcing set of $G$. Suppose for contradiction that $S$ contains a smaller zero forcing set $S'$ as a proper subset. If $S'$ is not a cover of $\mathcal{B}(G)$, 
then there exists a fort $F$ which does not contain any element of $S'$. In order for the first vertex $v$ of $F$ to be forced, at some timestep $v$ must be the only neighbor of some blue vertex outside $F$. However, since $F$ is a fort, any vertex outside $F$ which is adjacent to $v$ is also adjacent to another white vertex
in $F$. Thus, $v$ cannot be forced, which contradicts $S'$ being a zero forcing set. Therefore, $S'$ is a cover of $G$. However, this contradicts the assumption that $S$ is a minimal cover. 
\end{proof}

\noindent When studying the structure of minimal zero forcing sets, it is useful to consider how minimal zero forcing sets intersect. The following proposition concerns vertices that appear in every minimal zero forcing set of a given graph.
\begin{prop}
Let $G=(V,E)$ be a graph and $v\in V$. Every minimal zero forcing set of $G$ contains $v$ if and only if $v$ is an isolate. 
\end{prop}

\begin{proof}
Clearly, since isolates must be contained in every zero forcing set of $G$, they must also be contained in every minimal zero forcing set of $G$. Suppose a non-isolate vertex $v$ is contained in every minimal zero forcing set of $G$. Let $S$ be an arbitrary minimum zero forcing set of $G$ (and hence also a minimal zero forcing set). Since $S$ is minimal and $v$ is not an isolate, there are neighbors of $v$ that are not in $S$. Let $L$ be a chronological list of forces of $S$. If $v$ forces a vertex in $L$, then the terminus of the set of forces in $L$ is a minimum zero forcing set that does not contain $v$. If $v$ does not force a vertex in $L$, let $L'$ be a chronological list of forces that is identical to $L$ except that in the step where the last white neighbor $w$ of $v$ is forced by some vertex $u$, instead $v$ forces $w$. Then, the terminus of the set of forces in $L'$ is a minimum zero forcing set that does not contain $v$. In both cases, the terminus is a zero forcing set of $G$ that has the same cardinality as $S$ and is therefore minimum (and hence minimal), which contradicts that $v$ is contained in every minimal zero forcing set. 
\end{proof}
\noindent Finally, while there are many graphs with $\zbar(G) = \Z(G)$, there are also graphs with a large gap between $\zbar(G)$ and $Z(G)$. In the following proposition, we show that for a graph $G$ of order $n$, $\zbar(G) - \Z(G)$ can be $\Omega(n)$, and in fact can be almost equal to $n$.

\begin{prop}\label{prop:zbarMinusZ}
There are infinitely many graphs $G$ such that $\zbar(G)-\Z(G) = n-7$.
\end{prop}

\begin{proof}
Let $G_n$ be the graph $(2K_2)\vee P_{n-4}$ for each $n \geq 7$; see Figure \ref{fig:nminus7} for an illustration. Every vertex in $P_{n-4}$ together with one vertex from each $K_2$ forms a minimal zero forcing set of size $n-2$.
Since $G_n\not\cong K_m \cup kK_1$ for $k\geq 2$ and $m = n - k$, it follows from Proposition \ref{zbarnminus1} that $\zbar(G_n)\neq n-1$. Thus, $\zbar(G_n)=n-2$. On the other hand, every vertex in $2K_2$ together with a leaf in $P_{n-4}$ forms a minimum zero forcing set of size $5$. Thus, $\zbar(G)-\Z(G) = (n-2)-5=n-7$.
\end{proof} 
\begin{figure}[H]
\centering
\scalebox{1}{\includegraphics{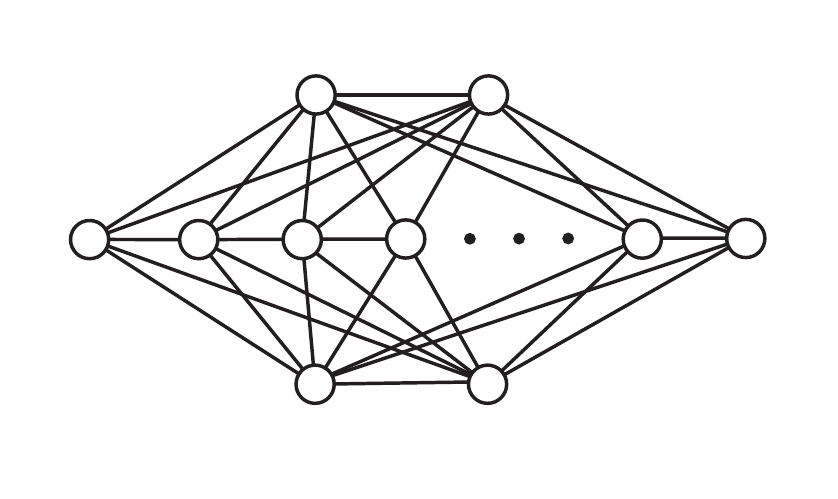}}
\caption{A graph $G_n$ on $n\geq $ vertices with $\zbar(G_n)-\Z(G_n)=n-7$.}\label{fig:nminus7}
\end{figure}

\section{Concluding remarks and future work}
\label{conclusion}
In this paper, we studied the structure, number, and maximum size of the minimal zero forcing sets of a graph. In Section \ref{sec:number}, we investigated the effect of several graph properties, like acyclicity and vertex transitivity, on the number of minimal zero forcing sets. Proposition~\ref{propproduct} showed a family of dense connected vertex transitive graphs that have an exponential number of minimal zero forcing sets. It is an open question to determine whether there is a family of sparse connected vertex transitive graphs that have exponentially many minimal zero forcing sets. To tackle this question, it could help to further understand the graphs with a polynomial number of minimal zero forcing sets. Therefore, determining which families of graphs have this property is also useful. 

It would also be interesting to further investigate when the minimal zero forcing sets of a graph can be found or counted in polynomial time. In particular, if a graph is known to have a polynomial number of minimal zero forcing sets, can all these sets be listed in polynomial time? Moreover, given a graph $G$ and a zero forcing set $B \subset V(G)$, when can the smallest minimal zero forcing set contained in $B$ be found in polynomial time? Note that answering this question for $B=V(G)$ is equivalent to finding the zero forcing number and is therefore NP-Hard.

In Section \ref{sec:max_min}, we focused on $\zbar(G)$, the maximum size of a minimal zero forcing set of a graph $G$. Generally, it seems nontrivial to find $\zbar(G)$, but the exact computational complexity is still unknown. Therefore, it would be interesting to determine whether $\zbar(G)$ can always be computed in polynomial time, or whether computing it is NP-Hard. Extablishing conditions which guarantee that $Z(G)=\zbar(G)$ is also a question of interest. Finally, in Proposition \ref{prop:zbarMinusZ}, we produced a family of graphs with $\zbar(G) - \Z(G) = n -7$; we leave it as an open question to determine the largest possible gap between $\zbar(G)$ and $\Z(G)$.

\end{document}